\pdfoutput=1
\documentclass[11pt]{article}

\usepackage[margin=1in]{geometry}
\IfFileExists{glyphtounicode.tex}{\input{glyphtounicode}\pdfgentounicode=1}{}
\usepackage{cmap}
\usepackage[T1]{fontenc}
\usepackage[utf8]{inputenc}
\usepackage{lmodern}
\usepackage{microtype}
\usepackage{amsmath,amssymb,amsthm,mathtools,mathrsfs,bm}
\usepackage{enumitem}
\usepackage{booktabs}
\usepackage{array}
\usepackage{cite}
\usepackage[bookmarks=false]{hyperref}
\hypersetup{
  colorlinks=true,
  linkcolor=blue,
  citecolor=blue,
  urlcolor=blue,
  pdftitle={Localized Centered Second-Chaos Operators},
  pdfauthor={Guangqian Zhao},
  pdfsubject={Localized centered second-chaos random operator estimates and pathwise cutoff convergence},
  pdfkeywords={centered Gaussian chaos, non-commutative Khintchine inequality, Schatten estimates, continuous frequency analysis, pathwise cutoff convergence, paracontrolled operators}
}
\newcommand{\doi}[1]{\href{https://doi.org/#1}{\textsc{doi}:\,\nolinkurl{#1}}}
\numberwithin{equation}{section}

\theoremstyle{plain}
\newtheorem{theorem}{Theorem}[section]
\newtheorem{proposition}[theorem]{Proposition}
\newtheorem{lemma}[theorem]{Lemma}
\newtheorem{corollary}[theorem]{Corollary}

\theoremstyle{definition}

\newtheorem{assumption}[theorem]{Assumption}

\newcommand{\R}{\mathbb R}

\newcommand{\E}{\mathbb E}

\newcommand{\cA}{\mathcal A}
\newcommand{\cB}{\mathcal B}
\newcommand{\cC}{\mathcal C}

\newcommand{\cE}{\mathcal E}
\newcommand{\cH}{\mathcal H}
\newcommand{\cK}{\mathcal K}
\newcommand{\cL}{\mathcal L}
\newcommand{\cT}{\mathcal T}
\newcommand{\Sch}{\mathfrak S}
\newcommand{\Dyd}{\mathbb D}
\newcommand{\Triads}{\mathfrak D}
\newcommand{\Mclass}{\mathfrak M}

\newcommand{\Xspace}{\mathbb X}
\newcommand{\Yspace}{\mathbb Y}
\newcommand{\dd}{\,\mathrm d}
\newcommand{\eps}{\varepsilon}
\newcommand{\wt}{\widetilde}

\newcommand{\Id}{\operatorname{Id}}

\newcommand{\Kdet}{K_{\mathrm{det}}}
\newcommand{\supp}{\operatorname{supp}}
\newcommand{\Sym}{\operatorname{Sym}}

\newcommand{\norm}[1]{\left\lVert #1\right\rVert}
\newcommand{\ip}[2]{\left\langle #1,#2\right\rangle}

\title{Localized Centered Second-Chaos Operators}
\author{Guangqian Zhao\thanks{E-mail: \texttt{zhaoguangqian@mail.ustc.edu.cn}}\\
School of Mathematical Sciences, University of Science and Technology of China,\\
Hefei, Anhui 230026, China}
\date{}

\begin{document}
\maketitle

\begin{abstract}
We prove a localized continuous-frequency operator estimate for centered Gaussian chaoses of order two.  The result applies to operator-valued centered second chaoses, including Wick-centered same-family variants, between Hilbert spaces.  In the model, two Gaussian frequency legs at scale $N$, an input leg at scale $Q$, and an output leg at scale $M$ are coupled through a soft incidence kernel; non-orthogonal Gaussian profiles are represented by covariance synthesis maps.

The proof combines four oriented flattenings, rectangular non-commutative Khintchine inequalities, soft-incidence Schatten bounds, and Sobolev--Besov dyadic summation.  The time lift gives $L^p$ operator convergence, while a Galerkin stabilization hypothesis gives pathwise full-cutoff convergence by the first Borel--Cantelli lemma.  Under $\mathcal G(N)\lesssim N^{-\Gamma}$ one obtains the window
\[
 \Gamma>\frac d2,
 \qquad
 s<\lambda+\Gamma-d,
 \qquad
 \max\{0,d-\Gamma\}<\sigma<\lambda+\Gamma-d.
\]
The theorem applies to the near-output Wick-centered branch of localized paracontrolled resonant products on $\mathbb R^d$.
\end{abstract}

\medskip
\noindent\textbf{Keywords.} Centered Gaussian chaos; non-commutative Khintchine inequality; Schatten estimates; continuous frequency analysis; soft incidence kernels; paracontrolled operators; pathwise cutoff convergence; power-law envelopes.

\medskip
\noindent\textbf{MSC 2020.} 60H15, 60H30, 35R60, 35Q40, 42B37.

\section{Introduction}
\label{sec:intro}

The discrete periodic counterpart of the estimate considered here has five layers: finite Wick centering, local Gaussian normalization, a four-flattening random tensor estimate, incidence counting, and finally time lift, dyadic summation, and Borel--Cantelli.  The same architecture has a localized continuous analogue on $\R^d$, but every discrete ingredient must be replaced by a continuous one.  The correspondence is summarized below.

\begin{center}
\begin{tabular}{@{}ll@{}}
\toprule
Periodic/discrete ingredient & Continuous full-space replacement \\
\midrule
independent Fourier coordinates & covariance Hilbert synthesis maps \\
finite random matrix & compact integral operator \\
lattice incidence $n=q+\ell+\eta$ & soft incidence $k(n-q-\ell-\eta)$ \\
finite flattening dimension & Schatten effective dimension \\
row/column degree counting & continuous Schur plus Hilbert--Schmidt interpolation \\
finite block stabilization & compact-shell Galerkin stabilization \\
first Borel--Cantelli on dyadic blocks & the same probability argument \\
\bottomrule
\end{tabular}
\end{center}

We use the following conventions throughout.  The notation $A\lesssim B$ means $A\le CB$, where $C$ may depend on the dimension, the fixed cutoff functions, the admissibility constants $c_{\rm ap},C_{\rm out}$, and the time horizon $T$, but never on the dyadic scales or on the Galerkin cutoff.  The notation $A\lesssim_\kappa B$ allows dependence on the displayed auxiliary parameter $\kappa$.  Dyadic scales are inhomogeneous and begin at one, and all dyadic sums are over $\Dyd$ unless otherwise stated.  For Hilbert spaces $U,V$, $\Sch_r(U,V)$ denotes the Schatten class, $\cK(U,V)$ the compact operators, and all Hilbert tensor products are completed Hilbert tensor products.  We also use the standard convention $\Sch_\infty(U,V)=\cL(U,V)$, with the operator norm.

The continuous replacement in the table is local in physical space: it uses the soft incidence kernel generated by a physical cutoff, not an unweighted translation-invariant global kernel on $\R^d$.  The estimate isolates the centered second-chaos component that remains after the deterministic covariance contraction has been separated.  In applications, the Gaussian frequency legs are encoded by covariance synthesis maps, the four Schatten flattening bounds are verified for the localized kernel, and the dyadic Sobolev--Besov summation gives the operator bound; Corollary~\ref{cor:paracontrolled-app} records this use for a near-output paracontrolled Duhamel branch.

The argument is in the spirit of the row/column and flattening estimates of Deng--Nahmod--Yue~\cite{DNY}.  The closest periodic/discrete mixed paracontrolled framework is the author's preprint~\cite{ZhaoRTDD}.  The present paper keeps only the centered second-chaos branch, replaces lattice incidence by a localized continuous-frequency soft incidence, and uses covariance synthesis maps for non-orthogonal smooth frequency profiles.

The proof proceeds through Wick centering, covariance pullback, four flattenings, soft-incidence Schatten bounds, dyadic assembly, and Borel--Cantelli convergence.

Smooth frequency projections are not assumed independent.  The two stochastic legs are represented by bounded maps from frequency $L^2$ spaces into covariance Hilbert spaces, and complete Wick centering leaves a second homogeneous Wiener-chaos object.  The abstract estimate is not tied to the later Volterra form: it starts from a finite-rank operator-valued centered second chaos
\[
 Z_{\mathbf H}=
 \sum_{i,j}g_i h_j H_{ij}:\cC\to\cE,
\]
or, in the same-family case, from
\[
 \sum_{i,j}H^{\rm sym}_{ij}\bigl(G_iG_j-\E[G_iG_j]\bigr).
\]
The deterministic input is Schatten control of the four Hilbert-space flattenings; the soft-incidence kernel and Volterra integration provide the continuous-frequency model and its paracontrolled realization.

The four flattenings are the leaves of the two-step row/column Khintchine tree:
\[
 (a,b,c)\to e,\qquad
 (a,b,e)\to c,\qquad
 (a,c)\to(b,e),\qquad
 (b,c)\to(a,e).
\]
After the continuous frequency variables are inserted, these become
\[
 (\ell,\eta,q)\to n,\qquad
 (\ell,\eta,n)\to q,\qquad
 (\ell,q)\to(\eta,n),\qquad
 (\eta,q)\to(\ell,n).
\]
The first two directions retain one high-frequency degree and one external shell; the mixed directions see only the input/output phase-space volumes.  After the admissible domination $Q,M\lesssim N$, this gives the final $M^{d/2}+Q^{d/2}$ profile rather than a product profile.

Only the completely centered branch is treated.  The final application concerns the near-output sector $M\le C_{\rm out}N$ of
\[
 [I_\lambda(w\prec_{\rm loc}\Psi_1)\circ_{\rm loc}\Psi_2]^\circ,
\]
where $^\circ$ denotes complete Wick centering.  Deterministic contractions and localization tails are left to separate estimates.  The pathwise step uses only dyadic summability and the first Borel--Cantelli lemma.

\subsection{Main theorem}
\label{subsec:intro-main}

The precise theorem is stated in Section~\ref{sec:global}.  The following condensed form records the main conclusion.

\begin{theorem}[Main theorem, condensed form]
\label{thm:intro-main}
Assume the localized continuous Gaussian model described in Section~\ref{sec:setup}, together with the time-regularity and cutoff convergence hypotheses used in Section~\ref{sec:global}.  Suppose that the centered dyadic amplitude obeys the power-law envelope
\[
 \mathcal G(N)\le C_GN^{-\Gamma},\qquad \Gamma>\frac d2.
\]
Let $\lambda>0$ and choose exponents $s,\sigma$ in the strict window
\[
 s<\lambda+\Gamma-d,\qquad
 \max\{0,d-\Gamma\}<\sigma<\lambda+\Gamma-d.
\]
Then the localized centered Galerkin operators are uniformly bounded in every finite $L^p(\Omega)$ as maps
\[
 \mathbb X_T^\sigma
 \longrightarrow
 C_TH^{s-\lambda}\cap L_T^1B^{\sigma-\lambda}_{2,\infty},
\]
where $\mathbb X_T^\sigma$ is the dyadic model of $C_TL^2\cap L_T^\infty B^\sigma_{2,\infty}$ defined in~\eqref{eq:Xspace}.  They converge in $L^p(\Omega)$ in the corresponding operator norm.  Under the finite-state Galerkin stabilization hypothesis, the full cutoff sequence also converges pathwise on one probability-one event.  The same operator estimate applies to same-family Wick squares after complete Wick centering and order-two decoupling.
\end{theorem}

The $L^p$ assertions use only Assumption~\ref{ass:time}; the finite-state Galerkin hypothesis is used only for full-sequence pathwise convergence.  All regularity inequalities are strict.

\subsection{Organization of the proof}
\label{subsec:organization}

The rest of this section fixes the localized kernel, amplitude class, covariance synthesis maps, and centered blocks.  Section~\ref{sec:flattening} proves the four-flattening chaos estimate and soft-incidence Schatten bounds.  Section~\ref{sec:global} gives the time lift, dyadic assembly, Borel--Cantelli upgrade, and power-law theorem.  The last subsection rewrites the result in localized Bony notation for the near-output paracontrolled branch.

\subsection{The localized model and centered blocks}
\label{sec:setup}

\subsubsection{Dyadic geometry and physical localization}

Let
\[
 \Dyd:=\{2^j:j\in\mathbb N_0\}
\]
with the usual inhomogeneous convention at scale one.  Fix constants
\[
 0<c_{\rm ap}<1,\qquad C_{\rm out}\ge1,
\]
and define the admissible dyadic set
\begin{equation}\label{eq:dyadic-set}
 \Triads:=\{(N,Q,M)\in\Dyd^3:Q\le c_{\rm ap}N,\ M\le C_{\rm out}N\}.
\end{equation}
Finite low-frequency exceptions may be added to $\Triads$ without changing any statement below, since they only modify the constants.  Thus the inhomogeneous scale one causes no separate endpoint issue.

Choose fixed real-valued smooth cutoffs
\[
 \rho,\chi_{\rm in},\chi_{\rm out}\in C_c^\infty(\R^d),
 \qquad
 \norm{\rho}_{\infty}+\norm{\chi_{\rm in}}_{\infty}
 +\norm{\chi_{\rm out}}_{\infty}\le3
\]
with
\[
 \supp\rho\subset\{\xi:r_0\le |\xi|\le R_0\}
\]
for fixed $0<r_0<R_0<\infty$.  Put
\[
 \rho_N(\xi)=\rho(\xi/N),\qquad
 \chi_Q(q)=\chi_{\rm in}(q/Q),\qquad
 \chi_M(n)=\chi_{\rm out}(n/M).
\]
Thus
\[
 \norm{\rho_N}_{L^2}\lesssim N^{d/2},\qquad
 \norm{\chi_Q}_{L^2}\lesssim Q^{d/2},\qquad
 \norm{\chi_M}_{L^2}\lesssim M^{d/2}.
\]
Throughout the deterministic incidence estimates, expressions such as $\rho_N^2$ and $\chi_Q^2$ may equivalently be read as absolute squares.  All estimates are invariant under replacing the cutoffs by their absolute values inside the Schur and Hilbert--Schmidt majorants.

Let $k\in L^1(\R^d)\cap L^2(\R^d)$.  In applications $k$ is typically the Fourier transform of a compactly supported physical localizer, hence is Schwartz and in particular belongs to $L^1\cap L^2$.  The factor
\[
 k(n-q-\ell-\eta)
\]
is the soft form of the incidence $n=q+\ell+\eta$.

All frequency variables are integrated with respect to Lebesgue measure.  The only properties of the incidence kernel used below are the $L^1$ Schur control and the $L^2$ Hilbert--Schmidt control.  The changes of variables
\[
 z=n-q-\ell-\eta,
 \qquad \eta=n-q-\ell-z,
 \qquad \ell=n-q-\eta-z,
\]
and their variants are affine maps with determinant $\pm1$.  Thus the continuous estimates below contain no hidden Jacobian loss relative to the soft-incidence profile.  We work with real Hilbert spaces to keep the notation aligned with real isonormal processes; the complex Fourier-space case follows by applying the estimates to the underlying real Hilbert spaces, with at most universal changes in constants.

\subsubsection{Normalized order-zero amplitudes}

Fix an even integer $K>4d$.  For a smooth amplitude
$m=m_{N,Q,M}(\ell,\eta,q,n)$ on a fixed enlargement of the dyadic support, set
\begin{equation}\label{eq:symbol-norm}
 \norm{m}_{\Mclass^K_{N,Q,M}}
 :=\max_{|\alpha|+|\beta|+|\gamma|+|\delta|\le K}
 N^{|\alpha|+|\beta|}Q^{|\gamma|}M^{|\delta|}
 \norm{\partial_\ell^\alpha\partial_\eta^\beta
 \partial_q^\gamma\partial_n^\delta m}_{L^\infty}.
\end{equation}
This is a uniform $C^K$ norm after the normalization
\[
 (u,v,x,y)=(\ell/N,\eta/N,q/Q,n/M).
\]

For each cutoff $\Lambda\in\Dyd\cup\{\infty\}$, let
$c_\Lambda^{(1)},c_\Lambda^{(2)}$ be smooth Galerkin multipliers satisfying
\[
 |c_\Lambda^{(j)}|\le1,
 \qquad
 c_\Lambda^{(j)}(\xi)=1\text{ for }|\xi|\le\Lambda,
 \qquad
 c_\Lambda^{(j)}(\xi)=0\text{ for }|\xi|\ge2\Lambda,
\]
and $c_\infty^{(j)}\equiv1$.  Define
\[
 \rho_{\Lambda,N}^{(j)}=\rho_N c_\Lambda^{(j)}.
\]
For the pathwise full-sequence results we use the standard rescaled dyadic Galerkin family: there are fixed smooth profiles $c^{(j)}$ such that $c_\Lambda^{(j)}(\xi)=c^{(j)}(\xi/\Lambda)$ for finite $\Lambda$.  In that fixed-profile situation, for each fixed $N$, the family $\{\rho_{\Lambda,N}^{(j)}\}_{\Lambda\in\Dyd\cup\{\infty\}}$ has only finitely many distinct restrictions to the shell $|\xi|\sim N$ and is eventually equal to $\rho_N$.  The $L^p$ convergence statements below only need convergence of the multipliers and synthesis maps; the finite-state property is used only for the pathwise full-cutoff upgrade.

Let $\Kdet(N)>0$ be a deterministic multiplier envelope.  For
$u=(t,s)\in\Delta_T:=\{0\le s\le t\le T\}$, define the frequency four-linear kernel
\begin{equation}\label{eq:base-kernel}
 \begin{aligned}
 H^0_{\Lambda,N,Q,M,u}(\ell,\eta,q,n)
 :={}&\Kdet(N)\rho_{\Lambda,N}^{(1)}(\ell)
 \rho_{\Lambda,N}^{(2)}(\eta)\chi_Q(q)\chi_M(n)\\
 &\times m_{\Lambda,N,Q,M,u}(\ell,\eta,q,n)
 k(n-q-\ell-\eta).
 \end{aligned}
\end{equation}
We impose the normalized size bound
\begin{equation}\label{eq:m-size}
 \sup_{\Lambda,N,Q,M,u}
 \norm{m_{\Lambda,N,Q,M,u}}_{\Mclass^K_{N,Q,M}}
 \le C_m.
\end{equation}

\subsubsection{Covariance synthesis maps and Wick centering}

Let $\cH_1,\cH_2$ be real separable Hilbert spaces with independent isonormal Gaussian processes
\[
 W_j:\cH_j\to L^2(\Omega),\qquad j=1,2.
\]
We use the standard Gaussian Hilbert space and multiple Wiener integral conventions; see, for example, \cite{Janson,Nualart}.
For each $N$ and time, let
\[
 S_{1,\Lambda,N,s}:L^2_\ell\to\cH_1,
 \qquad
 S_{2,\Lambda,N,t}:L^2_\eta\to\cH_2
\]
be bounded covariance synthesis maps.  Assume envelopes $\mathcal A_1(N),\mathcal A_2(N)>0$ such that
\begin{equation}\label{eq:S-size}
 \sup_{\Lambda,s}\norm{S_{1,\Lambda,N,s}}_{\rm op}
 \le C_S\mathcal A_1(N),
 \qquad
 \sup_{\Lambda,t}\norm{S_{2,\Lambda,N,t}}_{\rm op}
 \le C_S\mathcal A_2(N).
\end{equation}
The total centered amplitude is
\begin{equation}\label{eq:G-envelope}
 \mathcal G(N):=\Kdet(N)\mathcal A_1(N)\mathcal A_2(N).
\end{equation}

Let $\mathbf H^0_{\Lambda,N,Q,M,u}$ be the four-linear form associated with the kernel~\eqref{eq:base-kernel}.  Its covariance pullback is
\begin{equation}\label{eq:pullback}
 \wt{\mathbf H}_{\Lambda,N,Q,M,u}(a,b,f,g)
 :=\mathbf H^0_{\Lambda,N,Q,M,u}
 (S_{1,\Lambda,N,s}^*a,S_{2,\Lambda,N,t}^*b,f,g).
\end{equation}
For independent Gaussian legs, the centered block is defined by
\begin{equation}\label{eq:dec-block}
 \ip{B^{\circ}_{\Lambda,N,Q,M}(u)f}{g}
 :=I_{1,1}\bigl(\wt{\mathbf H}_{\Lambda,N,Q,M,u;f,g}\bigr),
\end{equation}
where $I_{1,1}$ is the decoupled first-by-first Wiener integral on $\cH_1\otimes_2\cH_2$.

If the two stochastic legs come from the same isonormal process $W$ on a covariance Hilbert space $\cH$, we instead define
\begin{equation}\label{eq:wick-block}
 \ip{B^{\circ}_{\Lambda,N,Q,M}(u)f}{g}
 :=I_2\bigl(\Sym\wt{\mathbf H}_{\Lambda,N,Q,M,u;f,g}\bigr).
\end{equation}
Here $\Sym$ denotes orthogonal symmetrization in the two covariance legs.  We fix the following finite-coordinate convention.  If $(e_i)$ is an orthonormal basis of the covariance Hilbert space, $G_i=W(e_i)$, and $a=(a_{ij})$ is a symmetric finite matrix, then
\[
 I_2\left(\sum_{i,j}a_{ij}e_i\otimes e_j\right)
 =\sum_{i,j}a_{ij}\bigl(G_iG_j-\E[G_iG_j]\bigr)
 =\sum_{i,j}a_{ij}(G_iG_j-\delta_{ij}).
\]
Equivalently, for a non-symmetric finite matrix one first replaces $a$ by $a^{\rm sym}_{ij}=(a_{ij}+a_{ji})/2$ and then applies the preceding formula.  Other common Wiener--It\^o normalizations differ only by fixed $2!$-type constants; throughout this paper those constants are absorbed into the universal constants in the estimates.

Thus, if $H^{\rm sym}_{ij}(f,g)$ are the coordinates of $\Sym\wt{\mathbf H}_{\Lambda,N,Q,M,u;f,g}$, then the same-family block is
\[
 \sum_{i,j}H^{\rm sym}_{ij}(f,g)\,\bigl(G_iG_j-\E[G_iG_j]\bigr).
\]
The subtraction is the full scalar covariance pairing in the covariance Hilbert space, not merely a diagonal subtraction in the original frequency variables.  Non-diagonal covariances in physical Fourier coordinates are first pulled back to an orthonormal basis of the covariance Hilbert space; the covariance matrix is encoded in the synthesis maps $S_{j,\Lambda,N,t}$.  No orthogonality of smooth frequency projectors is imposed.  Complete Wick centering is performed before the operator estimate is applied, and the same-family estimate follows from the Banach-valued order-two decoupling estimate in Lemma~\ref{lem:wick-decoupling}.  Since symmetrization only exchanges the two high covariance legs, the deterministic soft-incidence estimates below control the symmetrized same-family kernel by the same profiles.

\section{Four flattenings and continuous incidence bounds}
\label{sec:flattening}

\subsection{The abstract four-flattening inequality}

Let $\cA,\cB,\cC,\cE$ be real separable Hilbert spaces.  For a finite-rank four-linear form
$\mathbf H:\cA\times\cB\times\cC\times\cE\to\R$, define the four oriented flattenings by Riesz representation:
\begin{align}
 \ip{F_1(a\otimes b\otimes c)}{e}_{\cE}
 &=\mathbf H(a,b,c,e),
 &F_1&:\cA\otimes_2\cB\otimes_2\cC\to\cE,
 \label{eq:F1-def}\\
 \ip{F_2(a\otimes b\otimes e)}{c}_{\cC}
 &=\mathbf H(a,b,c,e),
 &F_2&:\cA\otimes_2\cB\otimes_2\cE\to\cC,
 \label{eq:F2-def}\\
 \ip{F_3(a\otimes c)}{b\otimes e}_{\cB\otimes_2\cE}
 &=\mathbf H(a,b,c,e),
 &F_3&:\cA\otimes_2\cC\to\cB\otimes_2\cE,
 \label{eq:F3-def}\\
 \ip{F_4(b\otimes c)}{a\otimes e}_{\cA\otimes_2\cE}
 &=\mathbf H(a,b,c,e),
 &F_4&:\cB\otimes_2\cC\to\cA\otimes_2\cE.
 \label{eq:F4-def}
\end{align}

We use the following standard Schatten ideal facts; see, for example, \cite{Simon,Pisier}.

\begin{lemma}[Schatten ideal, interpolation, and tensor regrouping]
\label{lem:operator-ideal}
Let $U,V,U',V'$ be Hilbert spaces.
\begin{enumerate}[label=\textup{(\roman*)},leftmargin=2.2em]
 \item If $T\in\cL(U,V)\cap\Sch_2(U,V)$ and $2\le r<\infty$, then
 \begin{equation}\label{eq:Schatten-interp}
  \norm{T}_{\Sch_r}
  \le \norm{T}_{\rm op}^{1-2/r}\norm{T}_{\Sch_2}^{2/r}.
 \end{equation}
 \item If $A\in\cL(V,V')$, $B\in\cL(U',U)$, and $T\in\Sch_r(U,V)$, then
 \begin{equation}\label{eq:ideal-two-sided}
  \norm{ATB}_{\Sch_r(U',V')}
  \le \norm{A}_{\rm op}\norm{T}_{\Sch_r(U,V)}\norm{B}_{\rm op}.
 \end{equation}
 \item The canonical regroupings of Hilbert tensor products, for instance
 $((\cA\otimes_2\cB)\otimes_2\cC)\simeq\cA\otimes_2\cB\otimes_2\cC$, are unitary.  Consequently Hilbert--Schmidt norms are unchanged by a mere reorientation of tensor legs.  Multiplication by unimodular factors on any frequency leg is also unitary and hence preserves every Schatten norm.
\end{enumerate}
\end{lemma}

\begin{proof}
For (i), let $(s_j(T))$ be the singular values.  Since
$\sum_j s_j(T)^r\le \norm{T}_{\rm op}^{r-2}\sum_js_j(T)^2$, taking the $r$-th root gives~\eqref{eq:Schatten-interp}.  Statement (ii) is the standard two-sided ideal property of Schatten classes, first for finite-rank operators and then by completion.  Statement (iii) follows from the defining universal property of Hilbert tensor products and from Fubini's theorem for the concrete $L^2$ models used in the frequency variables.
\end{proof}

\begin{lemma}[Finite compression of Schatten flattenings]
\label{lem:finite-compression}
Let $2\le r<\infty$ and let $F_1,F_2,F_3,F_4$ be the four flattenings associated with a four-linear form $\mathbf H$, with $F_j\in\Sch_r$.  Let $P_{\cA}^{(m)},P_{\cB}^{(m)},P_{\cC}^{(m)},P_{\cE}^{(m)}$ be finite-rank orthogonal projections increasing strongly to the identities.  Define the compressed form
\[
 \mathbf H_m(a,b,c,e)
 :=\mathbf H(P_{\cA}^{(m)}a,P_{\cB}^{(m)}b,
 P_{\cC}^{(m)}c,P_{\cE}^{(m)}e).
\]
If $F_{j,m}$ are the four flattenings of $\mathbf H_m$, then
\[
 \norm{F_{j,m}-F_j}_{\Sch_r}\longrightarrow0,
 \qquad j=1,2,3,4.
\]
\end{lemma}

\begin{proof}
For example,
\[
 F_{1,m}=P_{\cE}^{(m)}F_1
 (P_{\cA}^{(m)}\otimes P_{\cB}^{(m)}\otimes P_{\cC}^{(m)}),
\]
and the other orientations have the analogous left and right compressions.  If $T\in\Sch_r(U,V)$ and $P_m,Q_m$ are finite-rank orthogonal projections converging strongly to the identities, then $P_mTQ_m\to T$ in $\Sch_r$.  This follows first for finite-rank $T$ and then by approximation in $\Sch_r$.  Applying this observation to each flattening proves the claim.
\end{proof}

\begin{lemma}[Banach-valued decoupling for Wick squares]
\label{lem:wick-decoupling}
Let $B$ be a separable Banach space.  Let $(x_{ij})_{1\le i,j\le m}\subset B$, let $(G_i)_{i=1}^m$ be a standard real Gaussian family, and let $(g_i)_{i=1}^m$ and $(h_j)_{j=1}^m$ be two independent standard Gaussian copies.  Then, for every $1\le p<\infty$,
\begin{equation}\label{eq:wick-decoupling}
 \left\|\sum_{i,j=1}^m x_{ij}\bigl(G_iG_j-\E[G_iG_j]\bigr)\right\|_{L^p(\Omega;B)}
 \le C_{\rm dec}
 \left\|\sum_{i,j=1}^m x_{ij}g_ih_j\right\|_{L^p(\Omega;B)},
\end{equation}
where $C_{\rm dec}$ depends only on the order-two decoupling theorem, and not on $m$, on the Banach space $B$, or on $p$.  In particular, after symmetrizing the same-family kernel, a completely Wick-centered square may be estimated through its decoupled version.
\end{lemma}

\begin{proof}
This is the standard order-two Banach-valued decoupling inequality for Gaussian homogeneous chaoses; see, for instance, \cite[Chapter 3]{deLaPenaGine}.  The subtraction of $\E[G_iG_j]$ removes the zeroth-chaos component and leaves the second homogeneous Wiener-chaos projection.  The estimate is finite-dimensional here and is applied below with $B=\cK(\cC,\cE)$, or with the corresponding separable finite-shell operator space before passage to the limit.
\end{proof}

\begin{theorem}[Four-flattening second-chaos inequality]
\label{thm:four-abstract}
Let $2\le p\le r<\infty$, and suppose $F_1,F_2,F_3,F_4\in\Sch_r$.  Let $Z_{\mathbf H}:\cC\to\cE$ be the decoupled second-chaos operator associated with $\mathbf H$.  Then
\begin{equation}\label{eq:abstract-bound}
 \left\|\,\norm{Z_{\mathbf H}}_{\cC\to\cE}\,\right\|_{L^p(\Omega)}
 \le C r\sum_{j=1}^4\norm{F_j}_{\Sch_r},
\end{equation}
where $C$ is universal; the displayed linear factor $r$ is the only dependence on the chosen Schatten exponent in this estimate.  The map $\mathbf H\mapsto Z_{\mathbf H}$ extends continuously from finite-rank forms to all forms with $F_j\in\Sch_r$.  The same conclusion holds for the completely Wick-centered same-family chaos, up to the universal order-two decoupling constant.
\end{theorem}

\begin{proof}
Choose finite orthonormal systems $(a_i)$ and $(b_j)$ supporting a finite-rank approximation and write
\[
 Z=\sum_{i,j}g_ih_jH_{ij},\qquad H_{ij}\in\cL(\cC,\cE),
\]
with independent standard real Gaussian families.  Condition on $(g_i)$ and set
\[
 S_j(g)=\sum_i g_iH_{ij}.
\]
For each fixed realization of $(g_i)$, the rectangular non-commutative Khintchine inequality in the $h$-variables, in the Schatten-valued form of Lust-Piquard--Pisier~\cite{LPP} and Pisier~\cite{Pisier}, gives
\[
 \norm{Z(g,\cdot)}_{L^r_h\Sch_r}
 \le C\sqrt r\bigl(\norm{\mathscr S(g)}_{\Sch_r}+\norm{\mathscr T(g)}_{\Sch_r}\bigr),
\]
Here this is only used in the finite-dimensional rectangular form
\[
 \left\|\sum_j h_jA_j\right\|_{L^r(\Omega;\Sch_r)}
 \le C\sqrt r\left(
 \left\|\left(\sum_jA_jA_j^*\right)^{1/2}\right\|_{\Sch_r}
 +\left\|\left(\sum_jA_j^*A_j\right)^{1/2}\right\|_{\Sch_r}
 \right),
\]
with rectangular operators $A_j$ between Hilbert spaces.
The block row and block column are
\[
 \mathscr S(g):\cB\otimes_2\cC\to\cE,
 \qquad
 \mathscr S(g)(b_j\otimes c)=S_j(g)c,
\]
and
\[
 \mathscr T(g):\cC\to\cB\otimes_2\cE,
 \qquad
 \mathscr T(g)c=\sum_jb_j\otimes S_j(g)c.
\]
Taking $L^r$ in the $g$-variables, we estimate these two random Schatten norms once more by non-commutative Khintchine.  Write $\mathscr S(g)=\sum_i g_iT_i$, with
\[
 T_i:\cB\otimes_2\cC\to\cE,
 \qquad T_i(b_j\otimes c)=H_{ij}c.
\]
The two square functions generated by this block row are
\[
 \left(\sum_iT_iT_i^*\right)^{1/2}
 \quad\text{and}\quad
 \left(\sum_iT_i^*T_i\right)^{1/2},
\]
and the flattening identities are
\[
 F_1F_1^*=\sum_iT_iT_i^*,\qquad
 F_4^*F_4=\sum_iT_i^*T_i.
\]
Similarly, write $\mathscr T(g)=\sum_i g_iU_i$, where
\[
 U_i:\cC\to\cB\otimes_2\cE,
 \qquad
 U_i c=\sum_jb_j\otimes H_{ij}c.
\]
The second pair of square functions is
\[
 \left(\sum_iU_iU_i^*\right)^{1/2}
 \quad\text{and}\quad
 \left(\sum_iU_i^*U_i\right)^{1/2},
\]
and the corresponding flattening identities are
\[
 F_3F_3^*=\sum_iU_iU_i^*,\qquad
 F_2F_2^*=\sum_iU_i^*U_i.
\]
Thus the two-step row/column conditioning produces exactly the four oriented flattenings.  Each Khintchine application costs $C\sqrt r$, hence
\[
 \norm{Z}_{L^r(\Omega;\Sch_r)}
 \lesssim r\sum_{j=1}^4\norm{F_j}_{\Sch_r}.
\]
Since $p\le r$ and $\norm{T}_{\rm op}\le\norm{T}_{\Sch_r}$, this proves~\eqref{eq:abstract-bound}.

For general Schatten flattenings, choose finite-rank compressions of the four Hilbert legs.  Lemma~\ref{lem:finite-compression} gives convergence of the compressed flattenings in $\Sch_r$, and the finite-rank estimates above therefore make the compressed random operators Cauchy in $L^p(\Omega;\cK(\cC,\cE))$.  Because $\cC$ and $\cE$ are separable Hilbert spaces, $\cK(\cC,\cE)$ is separable in the operator norm.  The limit is independent of the chosen compressions.  This defines the extension of $\mathbf H\mapsto Z_{\mathbf H}$ and preserves~\eqref{eq:abstract-bound}.  The same-family statement follows by applying Lemma~\ref{lem:wick-decoupling}, with values in the separable compact-operator space, before the decoupled estimate.
\end{proof}

\begin{lemma}[Covariance/Bessel transfer]
\label{lem:bessel}
Let $U_1:\wt\cA\to\cA$ and $U_2:\wt\cB\to\cB$ be bounded, and define
\[
 \wt{\mathbf H}(\wt a,\wt b,c,e)
 :=\mathbf H(U_1\wt a,U_2\wt b,c,e).
\]
Then, for $2\le r\le\infty$, with $\Sch_\infty$ understood as the operator norm class,
\begin{equation}\label{eq:bessel}
 \max_j\norm{\wt F_j}_{\Sch_r}
 \le \norm{U_1}_{\rm op}\norm{U_2}_{\rm op}
 \max_j\norm{F_j}_{\Sch_r}.
\end{equation}
\end{lemma}

\begin{proof}
For $j=1,2$,
\[
 \wt F_j=F_j(U_1\otimes U_2\otimes\Id).
\]
For the mixed orientations,
\[
 \wt F_3=(U_2^*\otimes\Id)F_3(U_1\otimes\Id),
 \qquad
 \wt F_4=(U_1^*\otimes\Id)F_4(U_2\otimes\Id).
\]
Apply Lemma~\ref{lem:operator-ideal}\,(ii).
\end{proof}

\subsection{Continuous soft-incidence flattenings}

We first remove the order-zero amplitude and Galerkin factors.  Consider
\begin{equation}\label{eq:model-kernel}
 H(\ell,\eta,q,n)
 =A_N\rho_N(\ell)\rho_N(\eta)\chi_Q(q)\chi_M(n)
 k(n-q-\ell-\eta).
\end{equation}
Let $F_1,\dots,F_4$ be the integral operators with orientations
\[
 (\ell,\eta,q)\to n,\qquad
 (\ell,\eta,n)\to q,\qquad
 (\ell,q)\to(\eta,n),\qquad
 (\eta,q)\to(\ell,n).
\]

\begin{lemma}[Soft-incidence Schur fibers]
\label{lem:soft-schur}
Let the cutoffs be as in Section~\ref{sec:setup}, with uniform $L^\infty$ bounds and supports of volumes $O(N^d)$, $O(Q^d)$, and $O(M^d)$ at the corresponding scales.  Then the following bounds hold uniformly in the dyadic parameters:
\begin{align}
 \sup_n \iiint |k(z)|\rho_N(\ell)^2\chi_Q(q)^2
 \rho_N(n-q-\ell-z)^2\dd z\dd\ell\dd q
 &\lesssim \norm{k}_{L^1}N^dQ^d,
 \label{eq:schur-F1-fiber}\\
 \sup_q \iiint |k(z)|\rho_N(\ell)^2\chi_M(n)^2
 \rho_N(n-q-\ell-z)^2\dd z\dd\ell\dd n
 &\lesssim \norm{k}_{L^1}N^dM^d,
 \label{eq:schur-F2-fiber}\\
 \sup_{\eta,n}\iint |k(z)|\chi_Q(q)^2
 \rho_N(n-q-\eta-z)^2\dd z\dd q
 &\lesssim \norm{k}_{L^1}Q^d,
 \label{eq:schur-F3-first}\\
 \sup_{\ell,q,z}\int \rho_N(\eta)^2
 \chi_M(\ell+q+\eta+z)^2\dd\eta
 &\lesssim M^d.
 \label{eq:schur-F3-second}
\end{align}
The same estimates hold with the two high-frequency variables exchanged.
\end{lemma}

\begin{proof}
For~\eqref{eq:schur-F1-fiber}, integrate first in $(\ell,q)$ and use the volume bounds of the supports after the affine substitution $\eta=n-q-\ell-z$.  The determinant of this substitution is one, and $\rho_N\le C$, so the fiber is bounded by $C\norm{k}_{L^1}N^dQ^d$.  The proof of~\eqref{eq:schur-F2-fiber} is identical with the output shell replacing the input shell.  Estimate~\eqref{eq:schur-F3-first} follows by integrating in $q$ on a set of volume $O(Q^d)$ and then in $z$ against $|k|$.  Finally,
\[
 \int \rho_N(\eta)^2\chi_M(\ell+q+\eta+z)^2\dd\eta
 \le \norm{\rho_N}_{L^\infty}^2\norm{\chi_M}_{L^2}^2
 \lesssim M^d,
\]
uniformly in $(\ell,q,z)$.  The exchanged high-frequency estimate is obtained by exchanging the two high-frequency variables.
\end{proof}

\begin{proposition}[Four continuous operator-norm bounds]
\label{prop:op-flat}
For the model kernel~\eqref{eq:model-kernel},
\begin{align}
 \norm{F_1}_{\rm op}
 &\lesssim |A_N|\norm{k}_{L^1}N^{d/2}Q^{d/2},
 \label{eq:opF1}\\
 \norm{F_2}_{\rm op}
 &\lesssim |A_N|\norm{k}_{L^1}N^{d/2}M^{d/2},
 \label{eq:opF2}\\
 \norm{F_3}_{\rm op}+\norm{F_4}_{\rm op}
 &\lesssim |A_N|\norm{k}_{L^1}Q^{d/2}M^{d/2}.
 \label{eq:opF34}
\end{align}
\end{proposition}

\begin{proof}
We prove~\eqref{eq:opF1} and the $F_3$ estimate.  The remaining estimates follow by symmetry.

For $F_1$, set $z=n-q-\ell-\eta$, so $\eta=n-q-\ell-z$.  Weighted Cauchy--Schwarz with respect to $|k(z)|\dd z\dd\ell\dd q$ gives
\[
 \begin{aligned}
 |F_1f(n)|^2
 \lesssim{}& |A_N|^2
 \left(\iiint |k(z)|\rho_N(\ell)^2\chi_Q(q)^2
 \rho_N(n-q-\ell-z)^2\dd z\dd\ell\dd q\right)\\
 &\times\left(\iiint |k(z)|
 |f(\ell,n-q-\ell-z,q)|^2\dd z\dd\ell\dd q\right).
 \end{aligned}
\]
The first factor is bounded by~\eqref{eq:schur-F1-fiber}.  Integrating the second factor in $n$ and using the unit-Jacobian change of variables $\eta=n-q-\ell-z$ gives $\norm{k}_{L^1}\norm{f}_2^2$.  This proves~\eqref{eq:opF1}.

For $F_3$, write $\ell=n-q-\eta-z$.  Then
\[
 \begin{aligned}
 |F_3f(\eta,n)|^2
 \lesssim{}& |A_N|^2\rho_N(\eta)^2\chi_M(n)^2
 \left(\iint |k(z)|\chi_Q(q)^2
 \rho_N(n-q-\eta-z)^2\dd z\dd q\right)\\
 &\times\left(\iint |k(z)|
 |f(n-q-\eta-z,q)|^2\dd z\dd q\right).
 \end{aligned}
\]
The first parenthesis is bounded by~\eqref{eq:schur-F3-first}.  To integrate the second term in $(\eta,n)$, fix $(q,z)$ and use the unit-Jacobian change of variables
\[
 (\eta,n)\longmapsto (\eta,\ell),\qquad n=\ell+q+\eta+z.
\]
Then
\[
 \begin{aligned}
 &\iint \rho_N(\eta)^2\chi_M(n)^2
 \left(\iint |k(z)| |f(n-q-\eta-z,q)|^2\dd z\dd q\right)
 \dd\eta\dd n\\
 &\qquad=\iint |k(z)|\int |f(\ell,q)|^2
 \left(\int\rho_N(\eta)^2\chi_M(\ell+q+\eta+z)^2\dd\eta\right)
 \dd\ell\dd q\dd z.
 \end{aligned}
\]
The inner $\eta$-fiber is bounded by~\eqref{eq:schur-F3-second}; hence the last display is at most $C\norm{k}_{L^1}M^d\norm{f}_{L^2_{\ell,q}}^2$.  This yields~\eqref{eq:opF34} for $F_3$; $F_4$ is identical after exchanging $\ell$ and $\eta$.
\end{proof}

\begin{lemma}[Common Hilbert--Schmidt bound]
\label{lem:HS}
Each flattening of~\eqref{eq:model-kernel} has the same Hilbert--Schmidt norm, and
\begin{equation}\label{eq:HS}
 \norm{F_j}_{\Sch_2}
 \lesssim |A_N|\norm{k}_{L^2}
 N^{d/2}Q^{d/2}M^{d/2},\qquad j=1,2,3,4.
\end{equation}
\end{lemma}

\begin{proof}
The Hilbert--Schmidt norm is invariant under regrouping the four variables.  The square of this norm is bounded by
\[
 |A_N|^2\int \rho_N(\ell)^2\rho_N(\eta)^2
 \chi_Q(q)^2\chi_M(n)^2|k(n-q-\ell-\eta)|^2
 \dd\ell\dd\eta\dd q\dd n.
\]
Set $z=n-q-\ell-\eta$.  This affine change has determinant one.  Discarding one bounded high-frequency cutoff, the remaining integral is at most
$C|A_N|^2N^dQ^dM^d\norm{k}_{L^2}^2$.
\end{proof}

\begin{proposition}[Four Schatten flattenings]
\label{prop:Schatten}
For $2\le r<\infty$, let $C_k=\max\{\norm{k}_{L^1},\norm{k}_{L^2}\}$.  Then
\begin{align}
 \norm{F_1}_{\Sch_r}
 &\lesssim |A_N|C_kN^{d/2}Q^{d/2}M^{d/r},
 \label{eq:Sr1}\\
 \norm{F_2}_{\Sch_r}
 &\lesssim |A_N|C_kN^{d/2}M^{d/2}Q^{d/r},
 \label{eq:Sr2}\\
 \norm{F_3}_{\Sch_r}+\norm{F_4}_{\Sch_r}
 &\lesssim |A_N|C_kQ^{d/2}M^{d/2}N^{d/r}.
 \label{eq:Sr34}
\end{align}
\end{proposition}

\begin{proof}
Use Lemma~\ref{lem:operator-ideal}\,(i) and combine Proposition~\ref{prop:op-flat} with Lemma~\ref{lem:HS}.
\end{proof}

\subsection{Order-zero amplitude stability}

\begin{lemma}[Normalized Fourier expansion]
\label{lem:fourier-exp}
Let $K>4d$.  On the support of the four dyadic cutoffs, every amplitude $m$ with finite norm~\eqref{eq:symbol-norm} admits an expansion
\begin{equation}\label{eq:fourier-exp}
 m(\ell,\eta,q,n)
 =\sum_{\nu\in\mathbb Z^{4d}}c_\nu
 e^{i\theta_1\nu_1\cdot\ell/N}
 e^{i\theta_2\nu_2\cdot\eta/N}
 e^{i\theta_3\nu_3\cdot q/Q}
 e^{i\theta_4\nu_4\cdot n/M},
\end{equation}
where the $\theta_j$ depend only on a fixed normalized box and
\begin{equation}\label{eq:l1-fourier}
 \sum_\nu|c_\nu|
 \lesssim_{d,K}\norm{m}_{\Mclass^K_{N,Q,M}}.
\end{equation}
\end{lemma}

\begin{proof}
After the normalized change of variables, multiply by a fixed cutoff equal to one on the dyadic support and periodize on a fixed box in $\R^{4d}$.  The enlarged normalized boxes and the auxiliary cutoff are chosen once and for all, uniformly in the dyadic triple $(N,Q,M)$.  Integration by parts $K$ times gives
\[
 |c_\nu|\lesssim\langle\nu\rangle^{-K}
 \norm{m}_{\Mclass^K_{N,Q,M}}.
\]
Since $K>4d$, the Fourier coefficients are absolutely summable.
\end{proof}

Each Fourier mode in~\eqref{eq:fourier-exp} is a product of four unimodular multipliers, one on each Hilbert leg.  It therefore acts by unitary multiplication on the source and target of every oriented flattening.

\begin{corollary}[Order-zero four-flattening estimate]
\label{cor:symbol-flat}
Let
\[
 H^m=A_N\rho_N(\ell)\rho_N(\eta)\chi_Q(q)\chi_M(n)
 m(\ell,\eta,q,n)k(n-q-\ell-\eta).
\]
Then the bounds~\eqref{eq:Sr1}--\eqref{eq:Sr34} hold with the additional factor
$\norm{m}_{\Mclass^K_{N,Q,M}}$.
\end{corollary}

\begin{proof}
Expand $m$ by Lemma~\ref{lem:fourier-exp}.  Schatten norms are invariant under the four unitary modulations, and the triangle inequality together with~\eqref{eq:l1-fourier} completes the proof.
\end{proof}

\section{Time lift, dyadic assembly, and pathwise convergence}
\label{sec:global}

\subsection{Time-dependent centered block estimate}

We assume time regularity only through normalized coefficient and covariance-map increments.

\begin{assumption}[Time increments and cutoff stability]
\label{ass:time}
For every sufficiently small $0<\theta\le\theta_0$, there are constants $C_\theta$ and nonnegative scales
$\Theta_0(N),\Theta_1(N),\Theta_2(N)$ such that, uniformly in all cutoffs,
\begin{align}
 \norm{m_{\Lambda,N,Q,M,u}-m_{\Lambda,N,Q,M,v}}_{\Mclass^K_{N,Q,M}}
 &\le C_\theta\Theta_0(N)^\theta d_\Delta(u,v)^\theta,
 \label{eq:m-inc}\\
 \norm{S_{1,\Lambda,N,s}-S_{1,\Lambda,N,s'}}_{\rm op}
 &\le C_\theta\mathcal A_1(N)\Theta_1(N)^\theta|s-s'|^\theta,
 \label{eq:S1-inc}\\
 \norm{S_{2,\Lambda,N,t}-S_{2,\Lambda,N,t'}}_{\rm op}
 &\le C_\theta\mathcal A_2(N)\Theta_2(N)^\theta|t-t'|^\theta.
 \label{eq:S2-inc}
\end{align}
Here $d_\Delta((t,s),(t',s'))=|t-t'|+|s-s'|$.  We also assume the absorption condition: for every $\eps>0$ one can choose $\theta>0$ so small that
\begin{equation}\label{eq:time-absorb}
 \bigl(1+\Theta_0(N)+\Theta_1(N)+\Theta_2(N)\bigr)^\theta
 \lesssim_\eps N^\eps.
\end{equation}
Finally, for every fixed $(N,Q,M)$ the normalized amplitudes and synthesis maps converge in the corresponding $C^\theta$ norms as $\Lambda\to\infty$.  For the pathwise full-sequence theorem we impose the standard Galerkin-state condition: for every fixed dyadic triple, the family
\[
 (m_{\Lambda,N,Q,M},S_{1,\Lambda,N,\cdot},S_{2,\Lambda,N,\cdot})
\]
assumes at most $L_{\rm cut}$ distinct states, with $L_{\rm cut}$ independent of the scales, and is eventually equal to the limiting state.  For the standard rescaled dyadic Galerkin multipliers introduced above, this is a finite relative-position statement: for fixed $N$, the support of $\rho_N$ lies in $\{r_0N\le |\xi|\le R_0N\}$, so the ratio $\Lambda/N$ has only $O_{r_0,R_0}(1)$ relevant dyadic positions on this shell and the cutoff becomes identically one once $\Lambda\ge R_0N$.  If one only assumes convergence rather than eventual stabilization, the $L^p$ convergence statement remains valid, while the full-sequence pathwise theorem below should be read with this finite-state hypothesis included.
\end{assumption}

Polynomial time regularity of the normalized amplitudes and synthesis maps is a standard sufficient condition for \eqref{eq:m-inc}--\eqref{eq:time-absorb}: small Holder exponents absorb polynomial losses into $N^\eps$.

\begin{lemma}[Two-parameter Banach-valued chaining]
\label{lem:time-lift}
Let $X(u)$, $u\in\Delta_T$, take values in a separable Banach space $B$.  Suppose that for some $p_0\ge2$, $0<\theta\le1$, and constants $R_0,R_1$,
\[
 \sup_u\norm{X(u)}_{L^{p_0}(\Omega;B)}\le R_0,
 \qquad
 \norm{X(u)-X(v)}_{L^{p_0}(\Omega;B)}
 \le R_1d_\Delta(u,v)^\theta.
\]
If $p_0>4/\theta$, then $X$ has a separable continuous modification and, for every $2\le p\le p_0$,
\begin{equation}\label{eq:time-lift}
 \left\|\norm{X}_{C(\Delta_T;B)}\right\|_{L^p(\Omega)}
 \lesssim_{p,p_0,\theta,T}R_0+R_1.
\end{equation}
\end{lemma}

\begin{proof}
Use dyadic rational grids in the two-dimensional simplex.  At level $j$ there are $O_T(2^{2j})$ grid points.  Choose a parent in the coarser grid at distance $O(2^{-j})$, remaining inside the simplex.  The $L^{p_0}$ norm of the maximal increment at level $j$ is bounded by
\[
 C R_1 2^{-j\theta}2^{2j/p_0}.
\]
The series is summable because $p_0>4/\theta$.  Chaining on the rational skeleton and completion give a continuous modification.  Minkowski's inequality yields~\eqref{eq:time-lift}.
\end{proof}

\begin{theorem}[Uniform centered dyadic block]
\label{thm:block}
Assume~\eqref{eq:m-size}, \eqref{eq:S-size}, and Assumption~\ref{ass:time}.  For every finite $p\ge2$ and every $\eps>0$,
\begin{equation}\label{eq:block-bound}
 \sup_{\Lambda\in\Dyd\cup\{\infty\}}
 \left\|\,
 \norm{B^{\circ}_{\Lambda,N,Q,M}}_{C(\Delta_T;\cL(L^2_q,L^2_n))}
 \right\|_{L^p(\Omega)}
 \lesssim_{p,T,\eps}
 C_m C_S^2 C_k\mathcal G(N)
 N^{d/2+\eps}
 \bigl(Q^{d/2+\eps}+M^{d/2+\eps}\bigr).
\end{equation}
Under the Galerkin-state condition in Assumption~\ref{ass:time}, the same estimate holds with the supremum over $\Lambda$ inside the $L^p$ norm, with a constant depending additionally on the uniform finite-state bound $L_{\rm cut}$.  For each fixed dyadic triple,
\begin{equation}\label{eq:fixed-block-conv}
 B^{\circ}_{\Lambda,N,Q,M}
 \longrightarrow B^{\circ}_{\infty,N,Q,M}
 \quad\text{in }L^p\bigl(\Omega;C(\Delta_T;\cL(L^2_q,L^2_n))\bigr).
\end{equation}
\end{theorem}

\begin{proof}
At fixed $u$, Corollary~\ref{cor:symbol-flat}, Lemma~\ref{lem:bessel}, and Theorem~\ref{thm:four-abstract} give, for $r\ge p_0\ge p$,
\[
 \begin{aligned}
 \left\|\norm{B^{\circ}_{\Lambda,N,Q,M}(u)}_{\rm op}\right\|_{L^{p_0}}
 \lesssim{}& r C_m C_S^2 C_k\mathcal G(N)\\
 &\times\left[
 N^{d/2}Q^{d/2}M^{d/r}
 +N^{d/2}M^{d/2}Q^{d/r}
 +Q^{d/2}M^{d/2}N^{d/r}
 \right].
 \end{aligned}
\]
The exact three-term telescope for the covariance pullback is
\[
 \begin{aligned}
 \wt{\mathbf H}_u-\wt{\mathbf H}_v
 ={}&(\mathbf H^0_u-\mathbf H^0_v)
 (S_{1,u}^*\cdot,S_{2,u}^*\cdot,\cdot,\cdot)\\
 &+\mathbf H^0_v((S_{1,u}^*-S_{1,v}^*)\cdot,S_{2,u}^*\cdot,\cdot,\cdot)\\
 &+\mathbf H^0_v(S_{1,v}^*\cdot,(S_{2,u}^*-S_{2,v}^*)\cdot,\cdot,\cdot).
 \end{aligned}
\]
Using Assumption~\ref{ass:time}, the same four Schatten estimates hold for increments with the additional factor
\[
 C_\theta\bigl(1+\Theta_0(N)+\Theta_1(N)+\Theta_2(N)\bigr)^\theta
 d_\Delta(u,v)^\theta.
\]
Choose $\theta$ so that~\eqref{eq:time-absorb} costs at most $N^{\eps/4}$, then choose $p_0>\max\{p,4/\theta\}$ and
\[
 r\simeq p_0+\log(2+NQM).
\]
The factors $N^{d/r},Q^{d/r},M^{d/r}$ are uniformly bounded and $r\lesssim_{p,\eps}(NQM)^{\eps/4}$.  The finite-rank approximants and the Schatten limits take values in the separable compact-operator space $\cK(L^2_q,L^2_n)$, so the chaining lemma is applied first in $\cK(L^2_q,L^2_n)$ and then read as an operator-norm bound in $\cL(L^2_q,L^2_n)$.  Apply Lemma~\ref{lem:time-lift}.  Since $Q,M\lesssim N$, the mixed profile satisfies
\[
 Q^{d/2}M^{d/2}
 \lesssim N^{d/2}\bigl(Q^{d/2}+M^{d/2}\bigr).
\]
After renaming the small loss, this proves~\eqref{eq:block-bound}.

Under the Galerkin-state condition, the family of fixed-block random operators has at most $L_{\rm cut}$ distinct members.  More explicitly, if these members are $Y_1,\ldots,Y_L$ with $L\le L_{\rm cut}$, then
\[
 \left\|\sup_\Lambda\norm{B^\circ_{\Lambda,N,Q,M}}\right\|_{L^p}
 \le \left(\sum_{a=1}^{L}\norm{Y_a}_{L^p}^p\right)^{1/p}
 \le L_{\rm cut}^{1/p}\max_a\norm{Y_a}_{L^p},
\]
which gives the maximal estimate.  Independently of this finite-state hypothesis, the assumed $C^\theta$ convergence of amplitudes and synthesis maps, applied to the difference family, proves~\eqref{eq:fixed-block-conv}.
\end{proof}

\subsection{Dyadic assembly and $L^p$ operator convergence}

For $\sigma\in\R$, define the source sequence space
\begin{equation}\label{eq:Xspace}
 \norm{w}_{\Xspace_T^\sigma}
 :=\sup_{0\le t\le T}
 \left(\sum_{Q\in\Dyd}\norm{w_Q(t)}_{L^2}^2\right)^{1/2}
 +\sup_{Q\in\Dyd}Q^\sigma\norm{w_Q}_{L_T^\infty L^2}.
\end{equation}
For output exponents $\beta,\zeta\in\R$, define
\begin{equation}\label{eq:Yspace}
 \norm{u}_{\Yspace_T^{\beta,\zeta}}
 :=\sup_{0\le t\le T}
 \left(\sum_M M^{2\beta}\norm{u_M(t)}_{L^2}^2\right)^{1/2}
 +\int_0^T\sup_M M^\zeta\norm{u_M(t)}_{L^2}\dd t.
\end{equation}
Under standard finite-overlap Littlewood--Paley decompositions these are equivalent to
\[
 \Xspace_T^\sigma=C_TL^2\cap L_T^\infty B^\sigma_{2,\infty},
 \qquad
 \Yspace_T^{\beta,\zeta}=C_TH^\beta\cap L_T^1B^\zeta_{2,\infty}.
\]

Set
\begin{equation}\label{eq:a-eps}
 a^{(\eps)}_{N,Q,M}
 :=\mathcal G(N)N^{d/2+\eps}
 \bigl(Q^{d/2+\eps}+M^{d/2+\eps}\bigr).
\end{equation}
Define the deterministic Sobolev and Besov majorants
\begin{align}
 \mathfrak S_H^{(\eps)}(\beta,\sigma)
 &:=\left(\sum_M\left[
 M^\beta\sum_{\substack{N:\,M\le C_{\rm out}N}}
 \sum_{\substack{Q:\,Q\le c_{\rm ap}N}}
 a^{(\eps)}_{N,Q,M}Q^{-\sigma}
 \right]^2\right)^{1/2},
 \label{eq:SH}\\
 \mathfrak S_B^{(\eps)}(\zeta,\sigma)
 &:=\sum_M M^\zeta
 \sum_{\substack{N:\,M\le C_{\rm out}N}}
 \sum_{\substack{Q:\,Q\le c_{\rm ap}N}}
 a^{(\eps)}_{N,Q,M}Q^{-\sigma}.
 \label{eq:SB}
\end{align}

For a cutoff $\Lambda$, define the centered Volterra operator blockwise by
\begin{equation}\label{eq:assembled}
 (\cB^\circ_\Lambda w)_M(t)
 :=\sum_{N,Q}\int_0^t
 B^\circ_{\Lambda,N,Q,M}(t,s)w_Q(s)\dd s,
\end{equation}
with the sums restricted by~\eqref{eq:dyadic-set}.

\begin{lemma}[Finite-shell measurability and $L^p$ tails]
\label{lem:measurable-assembly}
Assume the majorant condition~\eqref{eq:majorants-finite}.  For each finite set $\mathcal F\subset\Triads$, let $\cB^\circ_{\Lambda,\mathcal F}$ be the operator obtained from~\eqref{eq:assembled} by retaining only triples in $\mathcal F$.  Then $\cB^\circ_{\Lambda,\mathcal F}$ is a strongly measurable finite-shell operator from $\Xspace_T^\sigma$ to $\Yspace_T^{\beta,\zeta}$.  For every fixed $p\ge2$ and every exhausting sequence of finite sets $\mathcal F_m\nearrow\Triads$, the sequence $\cB^\circ_{\Lambda,\mathcal F_m}$ is Cauchy in
\[
 L^p\bigl(\Omega;\cL(\Xspace_T^\sigma,\Yspace_T^{\beta,\zeta})\bigr).
\]
Its limit defines the full operator $\cB^\circ_\Lambda$, and the random variable
$\norm{\cB^\circ_\Lambda}_{\Xspace_T^\sigma\to\Yspace_T^{\beta,\zeta}}$ is measurable.
\end{lemma}

\begin{proof}
For finite $\mathcal F$, the operator is a finite sum of continuous-time compact-operator-valued blocks.  Each block is obtained as the operator-norm limit of finite-rank compressions in the separable spaces used in Theorem~\ref{thm:block}; hence the finite-shell operator and its norm are measurable.

Let $\mathcal R_{\Lambda,\mathcal E}$ denote the operator formed by summing only over a finite or cofinite set of triples $\mathcal E\subset\Triads$.  Repeating the proof of Theorem~\ref{thm:Lp-assembly} with the sums restricted to $\mathcal E$ gives, for every finite $p\ge2$,
\[
 \left\|\norm{\mathcal R_{\Lambda,\mathcal E}}_{\Xspace_T^\sigma\to\Yspace_T^{\beta,\zeta}}
 \right\|_{L^p(\Omega)}
 \lesssim_{p,T,\eps}
 \mathfrak S_{H,\mathcal E}^{(\eps)}(\beta,\sigma)
 +\mathfrak S_{B,\mathcal E}^{(\eps)}(\zeta,\sigma),
\]
where the two majorants are defined as in~\eqref{eq:SH}--\eqref{eq:SB}, but with the dyadic sums restricted to $\mathcal E$.  Since the full majorants are finite, their tails tend to zero as $\mathcal F\nearrow\Triads$.  Therefore the finite-shell operators form a Cauchy sequence in the displayed $L^p$ space.  The limit is the full operator by definition of the dyadic series in operator norm in $L^p$.  The operator norm of the limit is the $L^p$ limit, along a subsequence if necessary, of measurable finite-shell norm variables; hence it is measurable.  This lemma only supplies the $L^p$ construction and measurability.  The pathwise construction is proved later under the additional blockwise Borel--Cantelli bound of Lemma~\ref{lem:BC}.
\end{proof}

\begin{theorem}[$L^p$ centered operator]
\label{thm:Lp-assembly}
Fix $p\ge2$ and $\eps>0$.  If
\begin{equation}\label{eq:majorants-finite}
 \mathfrak S_H^{(\eps)}(\beta,\sigma)
 +\mathfrak S_B^{(\eps)}(\zeta,\sigma)<\infty,
\end{equation}
then $\cB^\circ_\Lambda$ is a well-defined random bounded operator
\[
 \cB^\circ_\Lambda:\Xspace_T^\sigma\to\Yspace_T^{\beta,\zeta},
\]
uniformly in $\Lambda$, and
\begin{equation}\label{eq:Lp-op}
 \left\|\norm{\cB^\circ_\Lambda}_{\Xspace_T^\sigma\to\Yspace_T^{\beta,\zeta}}
 \right\|_{L^p(\Omega)}
 \lesssim_{p,T,\eps}
 \mathfrak S_H^{(\eps)}(\beta,\sigma)
 +\mathfrak S_B^{(\eps)}(\zeta,\sigma).
\end{equation}
Moreover, as $\Lambda\to\infty$,
\begin{equation}\label{eq:Lp-conv}
 \cB^\circ_\Lambda\to\cB^\circ
 \quad\text{in }L^p\bigl(\Omega;\cL(\Xspace_T^\sigma,\Yspace_T^{\beta,\zeta})\bigr).
\end{equation}
\end{theorem}

\begin{proof}
For fixed $M,N,Q$, Theorem~\ref{thm:block} and Minkowski's inequality in the Volterra variable give
\[
 \begin{aligned}
 \left\|\norm{\int_0^t B^\circ_{\Lambda,N,Q,M}(t,s)w_Q(s)\dd s}_{C_TL^2}
 \right\|_{L^p}
 \lesssim_T a^{(\eps)}_{N,Q,M}\norm{w_Q}_{L_T^\infty L^2}.
 \end{aligned}
\]
The same bound holds in $L_T^1L^2$.  Since
\[
 \norm{w_Q}_{L_T^\infty L^2}
 \le Q^{-\sigma}\norm{w}_{\Xspace_T^\sigma},
\]
absolute summation in $(N,Q)$, followed by the $\ell^2_M$ norm of the shellwise $C_TL^2$ bounds and the $\ell^\infty_M$ Besov norm, gives~\eqref{eq:Lp-op}.  This is legitimate because
\[
 \sup_t\left(\sum_M M^{2\beta}\norm{u_M(t)}_2^2\right)^{1/2}
 \le
 \left(\sum_M M^{2\beta}\norm{u_M}_{C_TL^2}^2\right)^{1/2}.
\]
The $\ell^1_M$ majorant~\eqref{eq:SB} supplies the high-output tail required for convergence in $B^\zeta_{2,\infty}$; explicitly, $\int_0^T\sup_M M^\zeta\norm{u_M(t)}_2\dd t\le \sum_M M^\zeta\norm{u_M}_{L_T^1L^2}$.

For cutoff convergence, fix a finite dyadic set $\mathcal F$.  On $\mathcal F$, the block convergence~\eqref{eq:fixed-block-conv} and the finite sum give convergence in $L^p$ operator norm.  On the complement $\Triads\setminus\mathcal F$, the block estimate~\eqref{eq:block-bound} applied to the difference of two cutoff families gives
\[
 \left\|\norm{(\cB^\circ_\Lambda-\cB^\circ_{\Lambda'})_{\Triads\setminus\mathcal F}}
 _{\Xspace_T^\sigma\to\Yspace_T^{\beta,\zeta}}
 \right\|_{L^p}
 \lesssim
 \mathfrak S^{(\eps)}_{H,\Triads\setminus\mathcal F}
 +\mathfrak S^{(\eps)}_{B,\Triads\setminus\mathcal F}.
\]
The right-hand side tends to zero as $\mathcal F\nearrow\Triads$.  This finite-set/tail argument proves that $\cB^\circ_\Lambda$ is Cauchy in $L^p$ operator norm and identifies its limit with the operator built from the limiting blocks.  Hence~\eqref{eq:Lp-conv} follows.

Measurability follows from Lemma~\ref{lem:measurable-assembly}.
\end{proof}

\subsection{Borel--Cantelli and full-sequence pathwise convergence}

The almost-sure statement requires more than $L^p$ convergence: it requires one probability-one event on which all dyadic blocks are simultaneously controlled.

\begin{lemma}[Dyadic first-Borel--Cantelli device]
\label{lem:dyadic-BC-device}
Let $\mathcal I\subset\Dyd^3$ be countable.  Let $X_i\ge0$ be random variables and $A_i>0$ deterministic weights.  Suppose that for every finite $p\ge2$,
\[
 \norm{X_i}_{L^p(\Omega)}\le C_pA_i,
 \qquad i\in\mathcal I.
\]
Let $L_i\ge1$ be deterministic losses such that $\sum_{i\in\mathcal I}L_i^{-p}<\infty$ for some $p$.  Then there is a full-probability event $\Omega_*$ and a finite random variable $C_*(\omega)$ such that
\[
 X_i(\omega)\le C_*(\omega)L_iA_i,
 \qquad i\in\mathcal I,
 \quad \omega\in\Omega_*.
\]
No independence assumption on the family $(X_i)_{i\in\mathcal I}$ is required.
\end{lemma}

\begin{proof}
For $R>0$ define $E_i(R)=\{X_i>RL_iA_i\}$.  Chebyshev's inequality gives
$\mathbb P(E_i(R))\le C_p^pR^{-p}L_i^{-p}$, hence $\sum_i\mathbb P(E_i(R))<\infty$.  By the first Borel--Cantelli lemma, for each fixed $R$ only finitely many $E_i(R)$ occur almost surely.  Taking $R=1$ and absorbing the finitely many exceptional ratios $X_i/(L_iA_i)$ into $C_*(\omega)$ proves the claim.  The proof used only summability of probabilities.
\end{proof}

\begin{lemma}[Simultaneous pathwise block bound]
\label{lem:BC}
Assume, in addition to Assumption~\ref{ass:time}, the Galerkin-state finite-state condition stated there.  Fix a final loss $\eps>0$.  Choose $0<\eps_0<\eps$ and $\rho>0$ so that
\begin{equation}\label{eq:loss-hierarchy}
 \eps_0+3\rho<\eps.
\end{equation}
Then there is a full-probability event $\Omega_*$ and a finite random variable $C_*(\omega)$ such that, for every $\omega\in\Omega_*$, every dyadic triple in $\Triads$, and every cutoff,
\begin{equation}\label{eq:path-block}
 \norm{B^\circ_{\Lambda,N,Q,M}(\omega)}_{C(\Delta_T;\cL(L^2_q,L^2_n))}
 \le C_*(\omega)a^{(\eps)}_{N,Q,M}.
\end{equation}
No independence between distinct dyadic blocks is required.
\end{lemma}

\begin{proof}
Let
\[
 X_{N,Q,M}
 :=\sup_{\Lambda\in\Dyd\cup\{\infty\}}
 \norm{B^\circ_{\Lambda,N,Q,M}}_{C(\Delta_T;\cL)}.
\]
The maximal estimate in Theorem~\ref{thm:block}, run with loss $\eps_0$, gives for every finite $p$,
\[
 \norm{X_{N,Q,M}}_{L^p}
 \lesssim_p a^{(\eps_0)}_{N,Q,M}.
\]
Apply Lemma~\ref{lem:dyadic-BC-device} with the deterministic loss
$L_{N,Q,M}=N^\rho Q^\rho M^\rho$.  Choosing $p$ so that $p\rho>1$ makes
\[
 \sum_{(N,Q,M)\in\Triads}N^{-p\rho}Q^{-p\rho}M^{-p\rho}<\infty,
\]
because the variables are dyadic and start at scale one.  Hence, on one full-probability event,
\[
 X_{N,Q,M}(\omega)
 \le C_*(\omega)N^\rho Q^\rho M^\rho a^{(\eps_0)}_{N,Q,M}
\]
for all admissible triples.  Since $Q,M\lesssim N$,
\[
 N^\rho Q^\rho M^\rho\lesssim N^{3\rho},
\]
and~\eqref{eq:loss-hierarchy} absorbs this dyadic loss into the final $\eps$-majorant.  This proves~\eqref{eq:path-block}.
\end{proof}

\begin{theorem}[Pathwise centered operator and full cutoff convergence]
\label{thm:pathwise}
Assume the Galerkin-state finite-state condition in Assumption~\ref{ass:time} and
\[
 \mathfrak S_H^{(\eps)}(\beta,\sigma)
 +\mathfrak S_B^{(\eps)}(\zeta,\sigma)<\infty.
\]
Then on the full-probability event $\Omega_*$ of Lemma~\ref{lem:BC}, for every cutoff $\Lambda$ the series~\eqref{eq:assembled} defines a bounded operator
\[
 \cB^\circ_\Lambda(\omega):\Xspace_T^\sigma\to\Yspace_T^{\beta,\zeta},
\]
and
\begin{equation}\label{eq:path-uniform}
 \sup_{\Lambda\in\Dyd\cup\{\infty\}}
 \norm{\cB^\circ_\Lambda(\omega)}_{\Xspace_T^\sigma\to\Yspace_T^{\beta,\zeta}}
 \le C_*(\omega)
 \left[\mathfrak S_H^{(\eps)}+\mathfrak S_B^{(\eps)}\right].
\end{equation}
Moreover the entire Galerkin sequence, not merely a subsequence, satisfies
\begin{equation}\label{eq:path-conv}
 \cB^\circ_\Lambda(\omega)\to\cB^\circ(\omega)
 \quad\text{in }\cL(\Xspace_T^\sigma,\Yspace_T^{\beta,\zeta}).
\end{equation}
For every fixed input and output test vector, the limiting scalar pairing belongs to the centered second Wiener chaos.
\end{theorem}

\begin{proof}
Fix $\omega\in\Omega_*$.  The probability argument is now finished.  Replace the random block norm by the deterministic pathwise majorant~\eqref{eq:path-block} and repeat the assembly proof of Theorem~\ref{thm:Lp-assembly}.  This gives~\eqref{eq:path-uniform} directly and defines the operator on the full nonseparable source space without a density argument.

For convergence, choose a finite set of dyadic triples so that the pathwise Sobolev and Besov tails are arbitrarily small.  Every block in the finite set is stable for all sufficiently large Galerkin cutoffs.  Hence two large cutoffs agree on the finite set and differ only by the uniformly small tail.  This proves the Cauchy property of the full sequence and hence~\eqref{eq:path-conv}.

Finite partial pairings are centered second chaoses.  Their $L^2(\Omega)$ limit remains in the closed second-chaos subspace, proving the final assertion.
\end{proof}

\subsection{Power-law window and final theorem}

We now impose the power-law envelope relevant for dispersive and paracontrolled applications.  The dyadic decay of the deterministic multiplier and the covariance envelopes is encoded in a single exponent $\Gamma$, and the Sobolev--Besov summability conditions are verified explicitly from this exponent.

\begin{proposition}[Power-law summability]
\label{prop:power-law}
Assume that, for some $C_G>0$ and $\Gamma\in\R$,
\begin{equation}\label{eq:power-envelope}
 \mathcal G(N)\le C_G N^{-\Gamma}\qquad (N\in\Dyd).
\end{equation}
If
\begin{equation}\label{eq:power-window-beta-zeta}
 \Gamma>\frac d2,\qquad
 \sigma>\max\{0,d-\Gamma\},\qquad
 \beta<\Gamma-d,
 \qquad
 \zeta<\Gamma-d,
\end{equation}
then one can choose $\eps>0$ so small that
\begin{equation}\label{eq:power-majorants}
 \mathfrak S_H^{(\eps)}(\beta,\sigma)
 +\mathfrak S_B^{(\eps)}(\zeta,\sigma)<\infty.
\end{equation}
\end{proposition}

\begin{proof}
We suppress an arbitrarily small positive loss, which is restored at the end by decreasing the strict margins in~\eqref{eq:power-window-beta-zeta}.  The $M^{d/2}$ branch of the block profile is
\[
 N^{d/2-\Gamma}M^{d/2}Q^{-\sigma}.
\]
Since $\sigma>0$, the dyadic $Q$-sum is bounded.  Summing over $N\gtrsim M$ gives $M^{d/2-\Gamma}$ because $\Gamma>d/2$.  After multiplication by the output weight, the resulting powers are
\[
 M^{\beta+d-\Gamma}
 \qquad\text{and}\qquad
 M^{\zeta+d-\Gamma},
\]
which are summable under $\beta<\Gamma-d$ and $\zeta<\Gamma-d$.

For the $Q^{d/2}$ branch, the relevant summand is
\[
 N^{d/2-\Gamma}Q^{d/2-\sigma}.
\]
If $\sigma<d/2$, the $Q$-sum contributes $N^{d/2-\sigma}$, and the subsequent $N$-tail is controlled by $M^{d-\Gamma-\sigma}$ because $d-\Gamma-\sigma<0$.  The output exponents are then $\beta+d-\Gamma-\sigma$ and $\zeta+d-\Gamma-\sigma$, which are negative under $\beta,\zeta<\Gamma-d$ and $\sigma>0$.  If $\sigma\ge d/2$, the $Q$-sum is bounded up to a harmless endpoint logarithm, absorbed by the small $\eps$, and the $N$-tail is $M^{d/2-\Gamma}$; this is again summable after multiplication by $M^\beta$ or $M^\zeta$, since $\beta,\zeta<\Gamma-d$ is stronger than $\beta,\zeta<\Gamma-d/2$.  Finally choose $\eps>0$ so small that all strict inequalities in~\eqref{eq:power-window-beta-zeta} remain strict after the harmless replacements forced by the factors $N^{C\eps}Q^{C\eps}M^{C\eps}$.  This proves~\eqref{eq:power-majorants}.
\end{proof}

\begin{corollary}[Centered paracontrolled window]
\label{cor:lambda-window}
Let $\lambda>0$ be a deterministic smoothing gain and write
\[
 \beta=s-\lambda,
 \qquad
 \zeta=\sigma-\lambda.
\]
Under the power-law envelope~\eqref{eq:power-envelope}, the hypotheses of Proposition~\ref{prop:power-law} follow from
\begin{equation}\label{eq:lambda-window}
 \Gamma>\frac d2,
 \qquad
 s<\lambda+\Gamma-d,
 \qquad
 \max\{0,d-\Gamma\}<\sigma<\lambda+\Gamma-d.
\end{equation}
\end{corollary}

Here $\lambda$ is only a target reparametrization: the smoothing contribution is already included either in $\Kdet$ and hence in $\Gamma$, or in the output exponents $\beta=s-\lambda$ and $\zeta=\sigma-\lambda$.  The input interval is nonempty exactly when
\[
 \lambda+\Gamma-d>\max\{0,d-\Gamma\}.
\]

\begin{theorem}[Localized continuous centered second-chaos operator]
\label{thm:final}
Assume the continuous Gaussian setup of Section~\ref{sec:setup}, the order-zero symbol bound~\eqref{eq:m-size}, covariance bounds~\eqref{eq:S-size}, and the time-regularity and cutoff-convergence assumptions in Assumption~\ref{ass:time}.  Assume also the power-law envelope~\eqref{eq:power-envelope}.  Let $\lambda>0$ and let $s,\sigma\in\R$ satisfy the strict window~\eqref{eq:lambda-window}.  Then the following hold.
\begin{enumerate}[label=\textup{(\roman*)},leftmargin=2.4em]
\item For every finite $p\ge2$, the localized centered Galerkin operators
\[
 \cB^\circ_\Lambda:
 \Xspace_T^\sigma
 \longrightarrow
 C_TH^{s-\lambda}\cap L_T^1B^{\sigma-\lambda}_{2,\infty}
\]
are uniformly bounded in $L^p(\Omega)$ and converge in $L^p$ operator norm.  More precisely, after choosing $\eps>0$ below the strict margins in~\eqref{eq:lambda-window},
\[
 \sup_{\Lambda\in\Dyd\cup\{\infty\}}
 \left\|
 \norm{\cB^\circ_\Lambda}_{\Xspace_T^\sigma\to C_TH^{s-\lambda}\cap L_T^1B^{\sigma-\lambda}_{2,\infty}}
 \right\|_{L^p(\Omega)}<\infty,
\]
and
\[
 \cB^\circ_\Lambda\to\cB^\circ
 \quad\text{in }L^p\bigl(\Omega;\cL(\Xspace_T^\sigma,
 C_TH^{s-\lambda}\cap L_T^1B^{\sigma-\lambda}_{2,\infty})\bigr).
\]
\item Under the finite-state Galerkin condition in Assumption~\ref{ass:time}, there is one full-probability event $\Omega_*$ on which the whole Galerkin sequence converges pathwise in the same operator norm.  The same event controls all dyadic blocks simultaneously.
\end{enumerate}
The result applies to independent Gaussian legs and to same-family Wick squares after complete Wick centering and decoupling.  The finite-state hypothesis is used only in item~\textup{(ii)}.
\end{theorem}

\begin{proof}
Set $\beta=s-\lambda$ and $\zeta=\sigma-\lambda$.  Corollary~\ref{cor:lambda-window} and Proposition~\ref{prop:power-law} give the dyadic summability condition~\eqref{eq:majorants-finite} for a sufficiently small $\eps>0$.  The $L^p$ bound and $L^p$ cutoff convergence then follow from Theorem~\ref{thm:Lp-assembly}.  Under the additional Galerkin-state finite-state condition, the pathwise bound and full-sequence pathwise convergence follow from Theorem~\ref{thm:pathwise}.
\end{proof}

\subsection{Localized Bony notation and paracontrolled interpretation}

We finally record the near-output paracontrolled form of the estimate.

Fix a physical cutoff $\varphi\in C_c^\infty(\R^d)$ and set $k=\widehat\varphi$.  We use the localized Bony operations
\[
 f\prec_{\rm loc}g:=\varphi\,(f\prec g),\qquad
 f\circ_{\rm loc}g:=\varphi\,(f\circ g).
\]
When the same physical multiplier is inserted at the dyadic block level, the Fourier side of each block contains the soft incidence factor
\[
 \widehat\varphi(n-q-\ell-\eta)=k(n-q-\ell-\eta).
\]
Thus $\prec_{\rm loc}$ and $\circ_{\rm loc}$ are the usual Bony low--high and resonant operations followed by a fixed physical localization~\cite{Bony,BCD,GIP}.  The corollary records only the near-output sector $Q\le c_{\rm ap}N$, $M\le C_{\rm out}N$.

\begin{corollary}[Application to a near-output Wick-centered paracontrolled block]
\label{cor:paracontrolled-app}
Assume the hypotheses of Theorem~\ref{thm:final}.  Let $\Delta_Q$ and $\Delta_M$ be the Fourier multipliers associated with the fixed cutoffs $\chi_Q$ and $\chi_M$.  Let $\Psi^j_{\Lambda,N}$ be the Gaussian block synthesized by $S_{j,\Lambda,N,\cdot}$ and the isonormal process $W_j$.  Suppose that the dyadic multiplier of $I_\lambda(t,s)$, together with the localized product symbol, satisfies the order-zero amplitude hypothesis above with deterministic envelope $\Kdet(N)\lesssim N^{-\lambda}$ on the admissible near-output sector.  The near-output Wick-centered paracontrolled block is, at finite cutoff, the operator
\begin{equation}\label{eq:pc-operator-formal}
 \cT^{\circ,{\rm near}}_\Lambda(w)(t)
 :=
 \sum_{(N,Q,M)\in\Triads}
 \int_0^t
 \Delta_M
 \Bigl[
 I_\lambda(t,s)\bigl(\Delta_Qw(s)\,\Psi^1_{\Lambda,N}(s)\bigr)
 \circ_{\rm loc}
 \Psi^2_{\Lambda,N}(t)
 \Bigr]^\circ\,\dd s.
\end{equation}
Here $[\,\cdot\,]^\circ$ means complete Wick centering of the two Gaussian factors before the operator estimate is applied.  Equivalently, the $(N,Q,M)$ block of~\eqref{eq:pc-operator-formal} is precisely
\begin{equation}\label{eq:pc-equals-B}
 \int_0^t \cB^\circ_{\Lambda,N,Q,M}(t,s)w_Q(s)\,\dd s,
\end{equation}
with the centered kernel of Section~\ref{sec:setup}.

If
\begin{equation}\label{eq:pc-intervals}
 \Gamma>\frac d2,
 \qquad
 \sigma\in I_{\rm in}:=
 \bigl(\max\{0,d-\Gamma\},\,\lambda+\Gamma-d\bigr),
 \qquad
 s<\lambda+\Gamma-d,
\end{equation}
then, for every finite $p\ge2$,
\begin{equation}\label{eq:pc-map}
 \cT^{\circ,{\rm near}}_\Lambda:
 C_TL^2\cap L_T^\infty B^\sigma_{2,\infty}
 \longrightarrow
 C_TH^{s-\lambda}\cap L_T^1B^{\sigma-\lambda}_{2,\infty}
\end{equation}
is uniformly bounded in $L^p(\Omega)$ and converges in $L^p$ operator norm.  If the Galerkin-state finite-state condition holds, it also converges on one full-probability event along the whole Galerkin cutoff sequence.

If the Volterra multiplier contributes $N^{-\lambda}$ and the two Gaussian covariance envelopes contribute $N^{-\alpha_1}$ and $N^{-\alpha_2}$, then the power-law exponent is
\begin{equation}\label{eq:pc-gamma-model}
 \Gamma=\lambda+\alpha_1+\alpha_2.
\end{equation}
Thus the near-output centered Wick block~\eqref{eq:pc-operator-formal} maps the input regularity interval
\begin{equation}\label{eq:pc-model-interval}
 \max\{0,d-\lambda-\alpha_1-\alpha_2\}<\sigma
 <2\lambda+\alpha_1+\alpha_2-d,
 \qquad
 s<2\lambda+\alpha_1+\alpha_2-d.
\end{equation}
In the equal-gain case $\lambda=\alpha_1=\alpha_2=\alpha$, this becomes
\begin{equation}\label{eq:pc-equal-gain}
 \sigma\in\bigl(\max\{0,d-3\alpha\},\,4\alpha-d\bigr),
 \qquad
 s<4\alpha-d.
\end{equation}
The interval is nonempty exactly when $\alpha>2d/7$; in dimension $d=3$ this is $\alpha>6/7$.  If $3\alpha>d$, the lower endpoint is $0$; for $d=3$ and $\alpha=1$, one may take any $0<\sigma<1$ and any $s<1$.  Deterministic covariance contractions and far-output tails are not included in this corollary.
\end{corollary}

\begin{proof}
The localized dyadic decomposition of the near-output part of~\eqref{eq:pc-operator-formal} gives exactly the block identity~\eqref{eq:pc-equals-B}.  Hence~\eqref{eq:pc-map} is Theorem~\ref{thm:final} rewritten using the equivalence 
$\Xspace_T^\sigma=C_TL^2\cap L_T^\infty B^\sigma_{2,\infty}$ and the target identity in~\eqref{eq:Yspace}.  The explicit intervals~\eqref{eq:pc-model-interval} and~\eqref{eq:pc-equal-gain} are obtained by substituting~\eqref{eq:pc-gamma-model} into~\eqref{eq:pc-intervals}.
\end{proof}

\end{document}